\documentclass[a4paper, 12pt]{amsart}

\usepackage[colorlinks,citecolor = red, linkcolor=blue,hyperindex]{hyperref}
\usepackage{euscript,eufrak,verbatim}
\usepackage[psamsfonts]{amssymb}
\usepackage[usenames]{color}
\usepackage[curve]{xypic}
\usepackage{graphicx}
\usepackage{mathrsfs}
\usepackage{calligra,amsmath,amssymb}
\usepackage{float}
\usepackage{bbm}
 \usepackage{color}

\newtheorem{thm}{Theorem}[section]
\newtheorem{lem}[thm]{Lemma}

\newtheorem{prop}[thm]{Proposition}
\newtheorem{cor}[thm]{Corollary}
\newtheorem{defn}[thm]{Definition}
\newtheorem{rem}[thm]{Remark}

\def\R{\mathbb{R}}
\def\N{\mathbb{N}}

\def\E{\mathbb{E}}

\def\O{\mathcal{O}}

\def\X{\mathcal{X}}
\def\M{\mathcal{M}}

\def\C{\mathbb{C}}

\def\PP{\mathbb{P}}

\def\1{\mathbbm{1}}
\def\an {\text{\, and \,}}
\def\tr{\mathrm{tr}}

\def\Mat{\mathrm{Mat}}

\def\diag{\mathrm{diag}}

\def\U{\mathcal{U}}
\def\GS{\mathrm{GS}}
\def\erg{\mathcal{P}_{\mathrm{erg}}}
\def\inv{\mathcal{P}_{\mathrm{inv}}}

\def\XXint#1#2#3{{\setbox0=\hbox{$#1{#2#3}{\int}$ }
\vcenter{\hbox{$#2#3$ }}\kern-.6\wd0}}

\begin{document}

\title[Ergodic measures and infinite matrices of finite rank]{Ergodic measures and infinite matrices of finite rank}

%\date{\today}

\author{Yanqi Qiu}
\address%[authorlabel1]
{Yanqi QIU: CNRS, Institut de Math{\'e}matiques de Toulouse, Universit{\'e} Paul Sabatier, 118 Route de Narbonne, F-31062 Toulouse Cedex 9, France}

\email{yqi.qiu@gmail.com}

%\thanks{}

\begin{abstract}
Let $O(\infty)$ and $U(\infty)$ be the inductively compact infinite orthogonal group and infinite unitary group respectively. The classifications of ergodic probability measures with respect to the natural  group action of $O(\infty)\times O(m)$ on $\mathrm{Mat}(\mathbb{N}\times m, \mathbb{R})$ and  that of $U(\infty)\times U(m)$ on $\mathrm{Mat}(\mathbb{N}\times m, \mathbb{C})$ are due to Olshanski. The original proofs for these results are based on the asymptotic representation theory. In this note, by applying the Vershik-Kerov method, we propose a simple method for obtaining these two classifications,  making it accessible to pure probabilists.  
\end{abstract}


\maketitle

%\tableofcontents

\setcounter{equation}{0}

\section{Main results}\label{sec-main}
 
 Fix a positive integer $m\in\N$. Set
  \begin{align*}
 \M: = \Mat(\N\times m, \R)= \{ [X_{ij}]_{i\in\N, 1\le j \le m} | X_{ij}\in \R\}. 
 \end{align*}
Let $O(\infty)$ be the inductive limit group of the chain $$O(1) \subset O(2) \subset \cdots$$ of compact orthogonal groups. Equivalently, $O(\infty)$ is the group of infinite orthogonal matrices $u= [u_{ij}]$ such that $u_{ij} = \delta_{ij}$ when $i + j$ is large enough. Consider the natural group action of $O(\infty) \times O(m)$ on  $\M$ defined by 
\begin{align}\label{def-action}
((u, v), X) \mapsto u X v^{-1}, \quad u \in O(\infty), v \in O(m),  X \in \M. 
\end{align}
Let $\mathcal{P}_{\mathrm{erg}}(\M)$ denote the set of ergodic   $O(\infty) \times O(m)$-invariant Borel probability measures on $\M$, equipped with the induced weak topology.

Let $G : = [g_{ij}]_{i\in \N, 1 \le j \le m}$ be an infinite Gaussian random matrices on $\M$ such that $g_{ij}$'s are independent standard real Gaussian random variables. Let $\O$ be a random matrix sampled uniformly from $O(m)$ and independent of $G$.  
Define 
\begin{align*}
\Delta: =\{s = (s_1, \cdots, s_m) | s_1\ge \cdots\ge s_m \ge 0\}.
\end{align*}
 For any $s \in \Delta$, define $\mu_s$ as the probability distribution of the following random matrices 
$$
G\cdot \diag(s_1, \cdots, s_m) \cdot \O.
$$

\begin{thm}[{Olshanski \cite{Ol78, Ol84}}]\label{thm-r}
The map $s \mapsto \mu_s$ defines a homeomorphism between $\Delta$ and $\mathcal{P}_{\mathrm{erg}}(\M)$. 
\end{thm}

\begin{rem}
In the simplest case where $m=1$,  Theorem \ref{thm-r} reduces to the well-known Schoenberg theorem. 
\end{rem}

Similarly,  let $U(\infty)$ be  the inductive limit of the chain $$U(1) \subset U(2) \subset \cdots$$ of compact unitary groups and consider similar action as \eqref{def-action} of $U(\infty)\times U(m)$ on $\M^{\C} :   = \Mat(\N\times m , \C)$. Let $\mathcal{P}_{\mathrm{erg}}(\M^\C)$ denote the set of ergodic   $U(\infty) \times U(m)$-invariant Borel probability measures on $\M^\C$, equipped with the induced weak topology.

Let $G_{\C} = [g_{ij}^{\C}]_{i\in \N, 1\le j \le m}$ be an infinite Gaussian random matrices on $\M_\C$ such that $g^\C_{ij}$'s are independent standard complex  Gaussian random variables. Let $\mathcal{U}$ be a random matrix sampled uniformly from $U(m)$ and independently of $G_\C$. For any $s \in \Delta$, define $\mu_s^\C$ as the probability distribution of the following random matrices 
$$
G_\C \cdot \diag(s_1, \cdots, s_m) \cdot \U.
$$

\begin{thm}[{Olshanski \cite{Ol78, Ol84}}]\label{thm-c}
The map $s \mapsto \mu_s^\C$ defines a homeomorphism between $\Delta$ and $\mathcal{P}_{\mathrm{erg}}(\M_\C)$. 
\end{thm}

\begin{rem}
The reader is also referred to \cite{RV-Lie} for a recent related work on Olshanski spherical functions for infinite dimensional motion groups of fixed rank. 
\end{rem}

{\flushleft \bf Comments on the proof of Theorems \ref{thm-r} and \ref{thm-c}.}  

The proof of Theorem \ref{thm-c} is similar to that of Theorem \ref{thm-r}. Only the proof  of Theorem \ref{thm-r} will be detailed in this note.

In the case of  bi-orthogonally or bi-unitarily invariant measures on the space $\Mat(\N \times \N, \R)$ or $\Mat(\N\times\N, \C)$,  the ergodicity of an invariant measure is equivalent to the so-called Ismagilov-Olshanski multiplicativity of its Fourier transform, in particular, the ergodicity can be derived using the classical De Finetti Theorem from Ismagilov-Olshanski multiplicativity of its Fourier transform, see \cite{OkOl} and a  recent application of this method in \cite{BQ-sym} in non-Archimedean setting.   

However,  in our situations, there does not seem to be an analogue of  Ismagilov-Olshanski multiplicativity for the Fourier transforms of ergodic measures $\mu_s$ or $\mu_s^\C$. The proofs of the ergodicity for the measures $\mu_s$ or $\mu_s^\C$ require a new method. Two main ingredients for proving the ergodicity of $\mu_s$ are:  the mutual singularity between all measures $\mu_s$'s (derived from the strong law of large numbers) and an a priori ergodic decomposition formula due to Bufetov for invariant Borel probability measures with respect to a fixed  action of inductively compact group. 

Our method can also be applied to give a probably simpler proof, by avoiding the Harish-Chandra--Izykson-Zuber orbital integrals,  of the Olshanski and Vershik's approach to Pickrell's classification of unitarily ergodic Borel probability measures on the space of infinite Hermitian matrices. This part of work will be detailed elsewhere.  

\medskip

This research is supported by the grant IDEX UNITI-ANR-11-IDEX-0002-02, financed by Programme ``Investissements d'Avenir'' of the Government of the French Republic managed by the French National Research Agency.

%%%%%%%%%
%%%%%%%%%
%%%%%%%%%
%%%%%%%%%
%%%%%%%%%
%%%%%%%%%
%%%%%%%%%
%%%%%%%%%
%%%%%%%%%
%%%%%%%%%
%%%%%%%%%
%%%%%%%%%
%%%%%%%%%
%%%%%%%%%

\section{Preliminaries}\label{sec-pre}
\subsection{Notation}

Let $\X$ be a  Polish space. Denote by $\mathcal{P}(\X)$ the set of Borel probability measures on $\X$.  Let $G$ be a topological group and let $G$ acts on $\X$ by homeomorphisms. Let $\mathcal{P}_{\mathrm{inv}}^{G}(\X)$ denote the set of $G$-invariant Borel probability measures on $\X$. Recall that a measure $\mu \in \mathcal{P}_{\mathrm{inv}}^{G}(\X)$  is called  ergodic,  if  for any $G$-invariant Borel subset $\mathcal{A} \subset \mathcal{X}$, either $\mu(\mathcal{A}) =0$ or $\mu(\X\setminus \mathcal{A}) = 0$. Let $\mathcal{P}_{\mathrm{erg}}^{G}(\X)$ denote the set of ergodic $G$-invariant Borel probability measures on $\X$.  If the group action  is clear from the context, we also use the simplified notation $\mathcal{P}_{\mathrm{inv}}(\X)$ and $\mathcal{P}_{\mathrm{erg}}(\X)$. 

A sequence $(\mu_n)_{n\in\N}$ in $\mathcal{P}(\X)$ is said to converge weakly to $\mu\in \mathcal{P}(\X)$, and denoted by $\mu_n \Longrightarrow \mu$, if for any bounded continuous function $f$ on $\X$, we have 
\begin{align*}
\lim_{n\to\infty}\int_{\X} f d\mu_n  = \int_{\X} f d\mu.
\end{align*}   

Given any random variable $Y$, we denote by $\mathcal{L}(Y)$ its distribution.

Let $\M(\infty)$ be the subset of $\M$ consisting of matrices $X\in \M$ whose all but a finite number of entries vanish.  Let $\mu \in \mathcal{P}(\M)$, its Fourier transform is defined on $\M(\infty)$ by 
\begin{align*}
\widehat{\mu}(B)=\int_{\M} e^{i\cdot \tr (B^*\cdot X)} d\mu(X), \quad B \in \M(\infty).
\end{align*}

In what follows, for simplifying notation, for  $\lambda = (\lambda_1, \cdots, \lambda_m)\in \Delta$, we denote 
\begin{align}\label{def-D}
D_\lambda: = \diag(\lambda_1, \cdots, \lambda_m).
\end{align} 
When it is necessary, we also  identify $D_\lambda$ with an element of $\M(\infty)$ by adding infintely many $0$'s to make it a matrix in $\M(\infty)$

\begin{rem}\label{rem-diag}
Any $B\in \M(\infty)$ can be written in the form: 
\begin{align*}
B   =   u \cdot  \left[  \begin{array}{cccc}  \lambda_1& 0 & \cdots & 0 \\ 0& \lambda_2&  \cdots & 0  \\ \vdots& \vdots&  \ddots & \vdots  \\ 0&  0&  \cdots &  \lambda_m \\ 0&  0&  \cdots &  0 \\ \vdots&  \vdots&  \vdots &  \vdots\end{array}  \right]\cdot v,   
\end{align*}
where $ \lambda_1\ge \cdots \ge \lambda_m \ge 0$ and $u \in O(\infty), v\in O(m)$. 
 \end{rem}

The proof of the following lemma is elementary and is omitted here. 
\begin{lem}\label{lem-diag}
If $\mu\in\mathcal{P}_{\mathrm{inv}}(\M)$, then for $\mu$ is uniquely determined by $\widehat{\mu} (D_\lambda)$, where $\lambda$ ranges over $\Delta$.  

Given a sequence $(\mu_n)_{n\in\N}$ in $\inv(\M)$ and an element $\mu_\infty\in \inv(\M)$. The weak convergence $\mu_n \Longrightarrow \mu_\infty$ is equivalent to the uniform convergence $\widehat{\mu}_n(D_\lambda) \rightarrow \widehat{\mu}_\infty(D_\lambda)$  on compact subsets of $\Delta$. 
\end{lem}

\subsection{Ergodic measures for inductively compact groups}\label{sec-erg}

Here we briefly recall the Vershik-Kerov ergodic method.  Let $K(\infty)$ be the inductive limit of a chain $K(1) \subset K(2) \subset \cdots \subset K(n) \subset \cdots$ of compact metrizable groups.  Let  $K(\infty)$ act on  a Polish space $\X$ by homeomorphisms.

 Let $m_{K(n)}$ denote the normalized Haar measure on $K(n)$. Given a point $x\in \X$, let $m_{K(n)}(x)$ denote the image of $m_{K(n)}$ under the mapping $u 􏰁\mapsto u\cdot x$ from $K(n)$ to $\X$. 
 
\begin{defn}[Limit Orbital Measures]\label{def-L}
Define $\mathscr{L}^{K(\infty)} (\X)$ ($\mathscr{L}(\X)$ in short) the subset of Borel probability measures on $\X$, such that there exists a point $x\in \X$ and $m_{K(n)}(x) \Longrightarrow \mu$. 
\end{defn}
It is clear that $\mathscr{L} (\X)\subset \mathcal{P}_{\mathrm{inv}}(\X)$.

\begin{thm}[{Vershik \cite[Theorem 1]{Vershik-inf-group}}]\label{Ver-thm}
The following inclusion holds: 
$$
\mathcal{P}_{\mathrm{erg}}(\X) \subset \mathscr{L} (\X). 
$$ 
\end{thm}

We will also need an a priori ergodic decomposition formula due to Bufetov.  
\begin{thm}[{Bufetov \cite[Theorem 1]{Bufetov-erg-dec}}]\label{Buf-thm}
The set $\mathcal{P}_{\mathrm{inv}}(\X)$ is a Borel subset of $\mathcal{P}(\X)$.  For any $\nu\in\mathcal{P}_{\mathrm{inv}}(\X)$, there exists a  Borel probability $\overline{\nu}$ on $\mathcal{P}_{\mathrm{erg}}(\X)$ such that 
\begin{align}\label{erg-dec}
\nu  = \int\limits_{\mathcal{P}_{\mathrm{erg}}(\X)} \eta \,  d\overline{\nu} (\eta).  
\end{align}
\end{thm}  
\begin{rem}
Here the equality \eqref{erg-dec} means that for any Borel subset $\mathcal{A} \subset \X$, we have 
\begin{align}\label{dec-meaning}
\nu(\mathcal{A})  = \int\limits_{\mathcal{P}_{\mathrm{erg}}(\X)} \eta (\mathcal{A})   d\overline{\nu} (\eta).  
\end{align}
\end{rem}

\subsection{Haar random matrices from $O(N)$ or $U(N)$}

\subsubsection{How to sample Haar random matrices from $O(N)$ or $U(N)$?}\label{sec-haar}
We need the following well-known simple results.  Let $N\in\N$ be a fixed positive integer.  Let $G_N= (g_{ij})_{1\le i, j \le N}$ be a random $N\times N$ real matrix such that the entries $g_{ij}$'s are i.i.d standard real Gaussian random variables.  Similarly,  denote $G_N^{\C}$ the complex random matrix with i.i.d standard complex Gaussian random variables. 

For any $N\times N$ square real or complex matrix $A$,  let $\GS(A)$ be the matrix obtained from $A$ by doing the Gram-Schmidt orthogonalization procedure with respect to the columns of $A$.

Note that $\GS(G_N)\in O(N)$ almost surely. Moreover, for any given orthogonal matrix $O\in O(N)$, we have 
\begin{align*}
\GS(O\cdot G_{N}) = O\cdot \GS(G_N). 
\end{align*}

\begin{prop}\label{prop-haar}
The random matrix $\GS(G_N), \GS(G_N^{\C})$ are Haar random matrices from $O(N)$ and $U(N)$ respectively.  
\end{prop}

\begin{proof}
It suffices to prove that  the distribution $\GS(G_N)$ of  is invariant under left action by $O(N)$.  Fix  $O\in O(N)$.  By the invariance of Gaussian measure on $\R^N$, we have $O \cdot G_N \stackrel{d}{=} G_N$. This yields the desired equality 
\begin{align*}
O\cdot \GS(G_N)   = \GS(O\cdot G_{N})\stackrel{d}{=} \GS(G_{N}).
\end{align*}

Similar argument yields  that $\GS(G_N^{\C})$ is a Haar random matrix from $U(N)$.
\end{proof}

\subsubsection{Asymptotic of trunctations}\label{sec-trunc}

Let $Z^{(N)}$ and $Z_\C^{(N)}$ be a Haar random matrix from $O(N)$ and  $U(N)$ respectively. For any positive integer $S\le N$, we denote 
by $Z^{(N)}[S]$ and $Z_\C^{(N)}[S]$  the truncated upper-left $S\times S$ corner of $Z^{(N)}$ and $Z_\C^{(N)}$ respectively, that is, 
\begin{align*}
Z^{(N)}[ S]: =[Z^{(N)}_{ij}]_{1 \le i, j \le S} \an Z_\C^{(N)}[ S]: =[(Z_\C^{(N)})_{ij}]_{1 \le i, j \le S}. 
\end{align*}
Let 
$$
G^{(S)} = [g_{ij}]_{ 1\le i, j \le S} \an G_\C^{(S)} = [g_{ij}^{\C}]_{1\le i, j \le S}.
$$
where $g_{ij}$ (resp. $g^\C_{ij}$) are independent standard normal  real (resp. complex) random variables.

The following well-known result will be useful. 
\begin{prop}[Borel Theorem]\label{prop-trun}
As $N$ goes to infinity, the following weak convergences hold: 
\begin{align*}
\mathcal{L}( \sqrt{N} \cdot Z^{(N)}[S] ) \Longrightarrow G^{(S)} \an \mathcal{L}( \sqrt{N} \cdot Z_\C^{(N)}[S] ) \Longrightarrow G_\C^{(S)}. 
\end{align*}
\end{prop}

%%%%%%%%%
%%%%%%%%%
%%%%%%%%%
%%%%%%%%%
%%%%%%%%%
%%%%%%%%%
%%%%%%%%%
%%%%%%%%%
%%%%%%%%%
%%%%%%%%%
%%%%%%%%%
%%%%%%%%%
%%%%%%%%%
%%%%%%%%%

\section{Classification of $\mathcal{P}_{\mathrm{erg}}(\M)$}
\subsection{Singularity between $\mu_s$'s}

Recall that two Borel  probability measures $\sigma_1$ and $\sigma_2$ on $\X$ are called singular to each other, if there exists a Borel subset $\mathcal{A}\subset \X$ such that $\sigma_1(\mathcal{A}) = 1 - \sigma_2(\mathcal{A})=1$. 
\begin{prop}\label{prop-sing}
Probability measures from the family $\{\mu_s: s \in \Delta\}$ are mutually singular.  In particular, all the measures $\mu_s$'s are distinct. 
\end{prop}

\begin{rem}\label{rem-inj}
The map $s  \mapsto (\sum_{i =1}^m s_i^{2k})_{k\in\N}$ from $\Delta$ to $\R^\N$ is injective.   Indeed, first we have 
\begin{align*}
s_1 = \lim_{k\to\infty} (\sum_{i=1}^m s_i^{2k})^{1/2k}. 
\end{align*}
Then, the sequence $(\sum_{i =2}^m s_i^{2k})_{k\in\N}$ is known and so is $s_2$. Continue this procedure, we see that the sequence  $(\sum_{i =1}^m s_i^{2k})_{k\in\N}$ determines $s$ uniquely. 
\end{rem}

For any $n\in\N$, let $$C_n: \M \rightarrow  \Mat(n \times m, \R)$$ be the map sending any $X \in \M$ to its upper $n\times m$-corner. Equivalently, 
$$
C_n(X) = \diag(\underbrace{1, \cdots, 1}_{\text{$n$ times}},  0, \cdots)\cdot X. 
$$
In particular, we have $C_n(XY) = C_n(XY)$ once $XY$ is well-defined.

\begin{prop}\label{prop-as}
For any $s\in\Delta$ and any $k\in\N$, we have 
\begin{align}\label{tr-as}
\lim_{n\to\infty}\tr  \Big(  \left[\frac{( C_n(X))^* C_n(X)}{n}\right]^k \Big) = \sum_{i=1}^m s_i^{2k}, \text{ for $\mu_s$-a.e. $X\in \M$.}
\end{align}
\end{prop}

\begin{proof}
Since $C_n(GD_s\O)  = C_n(G) D_s \O$, we have 
\begin{align*}
C_n(G D_s \O))^* C_n(GD_s\O) = \O^* D_sC_n(G)^* C_n(G)D_s  \O.
\end{align*}
The random matrix $C_n(G)^* C_n(G)$ is of size $n\times n$. Let $1\le i, j \le n$, then the $(i, j)$-entry of $C_n(G)^* C_n(G)$  is 
\begin{align*}
[C_n(G)^* C_n(G)]_{ij} = \sum_{l=1}^n g_{li} g_{lj}. 
\end{align*}
By the strong  law of large numbers, we have  
\begin{align*}
\frac{[C_n(G)^* C_n(G)]_{ij}}{n} \xrightarrow[n\to\infty]{a.s.}\delta_{ij}. 
\end{align*} 
It follows that 
\begin{align*}
\left[  \frac{\O^* D_sC_n(G)^* C_n(G)D_s  \O}{n}\right]^k\xrightarrow[n\to\infty]{a.s.} (\O^* D_s^2  \O)^k  = \O^* D_s^{2k} \O. 
\end{align*}
As a consequence, we have
\begin{align}\label{rm-as}
\tr  \Big(  \left[  \frac{\O^* D_sC_n(G)^* C_n(G)D_s  \O}{n}\right]^k \Big) \xrightarrow[n\to\infty]{a.s.}  \tr (D_s^{2k}) = \sum_{i=1}^m s_i^{2k}. 
\end{align}
By the definition of $\mu_s$,   \eqref{rm-as} implies the desired assertion  \eqref{tr-as}. 
\end{proof}

\begin{proof}[Proof of Proposition \ref{prop-sing}]
For any $s\in \Delta$, we define a subset $\mathcal{A}_s\subset\M$ by 
\begin{align}\label{def-As}
\mathcal{A}_s: = \left\{ X \in \M \Big| \lim_{n\to\infty}\tr  \Big(  \left[\frac{( C_n(X))^* C_n(X)}{n}\right]^k \Big) = \sum_{i=1}^m s_i^{2k}, \forall k \in \N\right\}. 
\end{align}
By Remark \ref{rem-inj}, the Borel subsets $\mathcal{A}_s$'s are mutually disjoint.  Proposition \ref{prop-as} implies that  $\mu_s(\mathcal{A}_s) = 1$ for all $s \in \Delta$.  This implies that all the measures $\mu_s$'s are mutually singular. 
\end{proof}

\subsection{Tightness condition}
\begin{prop}\label{prop-tight}
Let $(s^{(n)})_{n\in\N}$ be a sequence  in $\Delta$. Then the sequence $(\mu_{s^{(n)}})_{n\in\N}$ is tight if and only if $\sup_{n\in\N} s_1^{(n)} < \infty.$
\end{prop}

\begin{proof}
Note that for any $s, \lambda \in \Delta$, we have 
\begin{align*}
\widehat{\mu}_s(D_\lambda)  = \E[\exp(i \cdot \tr (D_\lambda C_m(G) D_s \O) ]. 
\end{align*}
From this, it is clear that $s \mapsto \mu_s$ is continuous. Thus the compactness of the set $\{s \in \Delta| s_1 \le \sup_{n\in\N} s_1^{(n)} \}$  implies the tightness of  the sequence $(\mu_{s^{(n)}})_{n\in\N}$.

Conversely, let $(\mu_{s^{(n)}})_{n\in\N}$ be a tight sequence. Assume by contradiction that $\sup_{n\in\N} s_1^{(n)} =\infty$. Then we may find a sequence $n_1< n_2 < \cdots$ of positive integers, such that $ \lim_{k\to\infty} s_1^{(n_k)} = \infty$ and there exists $\mu\in \inv(\M)$ with
$$
 \mu_{s^{(n_k)}} \Longrightarrow \mu. 
$$
By independence between $C_m(G)$ and $\O$,  we have
\begin{align*}
&\widehat{\mu}_{s^{(n_k)}}(D_\lambda)  = \E\Big[\exp(i \cdot \sum_{l, j =1}^m \lambda_l s_j^{(n_k)}  g_{lj} \O_{j l})\Big]
\\
=& \E\Big[\prod_{l, j=1}^m \exp\Big(- \frac{\lambda_l^2 (s_j^{(n_k)})^2 (\O_{j l})^2}{2}\Big)\Big] 
\\
 =&  \E\Big[ \exp\Big(- \sum_{l, j=1}^m  \frac{ \lambda_l^2 (s_j^{(n_k)})^2 (\O_{j l})^2}{2}\Big)\Big] \le  \E\Big[ \exp\Big(-   \frac{ \lambda_1^2 (s_1^{(n_k)})^2 (\O_{11})^2}{2}\Big)\Big]. 
\end{align*}
Since $\O_{11} \ne 0$ a.s. and $\lim_{k\to\infty} s_1^{(n_k)} = \infty$,  for any $\lambda$ such that  $\lambda_1 > 0$, we have 
\begin{align*}
\widehat{\mu}(D_\lambda)  = \lim_{k\to\infty} \widehat{\mu}_{s^{(n_k)}}(D_\lambda)  = 0.
\end{align*}
By the uniform continuity of the Fourier transform $\widehat{\mu}(D_\lambda)$, we  get $\widehat{\mu}(D_{(0, \cdots, 0)})  = 0$. This obviously contradicts to the elementary fact that  $\widehat{\mu}(D_{(0, \cdots, 0)}) =1$. Hence  we must have $\sup_{n\in\N} s_1^{(n)} <\infty$. 

The proof of Proposition \ref{prop-tight} is completed. 
\end{proof}

\begin{cor}\label{cor-homeo}
The map $s \mapsto \mu_s$ induces a homeomorphism between $\Delta$ and $\{\mu_s: s\in \Delta\}$. 
\end{cor}
\begin{proof}
From above, the map $s \mapsto \mu_s$ is a continous bijection between $\Delta$ and $\{\mu_s: s\in \Delta\}$. 
 We only need to show the converse map is also continuous. Assume that $(s^{(n)})_{n\in\N}$ is a sequence in $\Delta$ and $s^{(\infty)} \in \Delta$ such that $\mu_{s^{(n)}} \Longrightarrow \mu_{s^{(\infty)}}$. By Proposition  \ref{prop-tight}, we have $\sup_{n\in\N} s_1^{(n)} < \infty$. Since the set $\{s \in \Delta| s_1 \le \sup_{n\in\N} s_1^{(n)} \}$ is compact, we only need to show that the sequence $(s^{(n)})_{n\in\N}$ has a  unique accumulation point. Let $s'$ be any accumulation point of the sequence $(s^{(n)})_{n\in\N}$. Then there exists a subsequence $(s^{(n_k)})_{n\in\N}$ that converges to $s'$. By continuity of the map $s\mapsto \mu_s$, we have $\mu_{s^{(n_k)}} \Longrightarrow \mu_{s'}$. It follows that $\mu_{s'} = \mu_{s^{(\infty)}}$ and hence $s' = s^{(\infty)}$. Thus $s^{(\infty)}$ is the unique accumulation point of the sequence $(s^{(n)})_{n\in\N}$, as desired.  
\end{proof}

\subsection{Limit orbital measures}
Recall the definition \eqref{def-L}, we denote 
$$\mathscr{L}(\M) : =\mathscr{L}^{O(\infty) \times O(m)}(\M).$$

\begin{prop}\label{prop-incl}
$\mathscr{L}(\M)\subset \{\mu_s: s \in \Delta\}.$
\end{prop}

Let us postpone the proof of Proposition \ref{prop-incl} to the next section and proceed with the proof of Theorem \ref{thm-r}.  

\begin{proof}[Proof of Theorem \ref{thm-r}]
By Corollary \ref{cor-homeo}, it suffices to prove that 
\begin{align}\label{erg-list}
\erg(\M) =\{\mu_s: s \in \Delta\}.
\end{align}
 By Theorem \ref{Ver-thm} and Proposition \ref{prop-incl},  we  have $ \erg(\M) \subset \{\mu_s: s \in \Delta\}$. 
 Hence we only need to show that for any $s_0\in \Delta$, we have $\mu_{s_0}\in \erg(\M)$.  By Theorem \ref{Buf-thm}, there exists a Borel probability measure $\overline{\nu}_{s_0}$ on $\erg(\M)$, such that 
\begin{align*}
\mu_{s_0} = \int\limits_{\erg(\M)} \eta \, d\overline{\nu}_{s_0}(\eta). 
\end{align*}
Denote by $j$ the inclusion map  $j:  \erg(\M) \hookrightarrow \{\mu_s: s \in \Delta\}$. Let $\Delta_{\mathrm{erg}}$ be the subset of $\Delta$ such that 
$$
j ( \erg(\M)) =  \{\mu_s: s \in \Delta_{\mathrm{erg}}\}. 
$$
Then $\Delta_{\mathrm{erg}}$ is a Borel subset and by Corollary \ref{cor-homeo},  there exists a Borel probability measure  $\widetilde{\nu}_{s_0}$ on $\Delta_{\mathrm{erg}}$ such that 
\begin{align}\label{s-dec-s}
\mu_{s_0} = \int\limits_{ \Delta_{\mathrm{erg}}} \mu_s \, d \widetilde{\nu}_{s_0}(s). 
\end{align}
Recall the definition \eqref{def-As} of the subset $\mathcal{A}_{s} \subset \M$. By \eqref{dec-meaning}, the equality \eqref{s-dec-s} implies that 
\begin{align*}
 \mu_{s_0}(\mathcal{A}_{s_0}) = \int\limits_{ \Delta_{\mathrm{erg}}} \mu_s (\mathcal{A}_{s_0}) \, d \widetilde{\nu}_{s_0}(s), 
\end{align*}
which in turn implies that 
\begin{align*}
1 = \int\limits_{ \Delta_{\mathrm{erg}}} \1_{s  = s_0} d \widetilde{\nu}_{s_0}(s). 
\end{align*}
It follows that $ \widetilde{\nu}_{s_0} = \delta_{s_0}$, where $\delta_{s_0}$ is the Dirac measure  on the point $s_0$. Since $\widetilde{\nu}_{s_0}$ is a probability measure on $\Delta_{\mathrm{erg}}$, we must have $s_0\in  \Delta_{\mathrm{erg}}$. Hence we get the desired relation $\mu_{s_0}\in \erg(\M)$. The proof of Theorem \ref{thm-r} is completed. 
\end{proof}

\section{Limit orbital measures are $\mu_s$'s}

The following lemma will be used. 

\begin{lem}\label{lem-2-seq}
Let $(X_n)_{n\in \N}, (Y_n)_{n\in\N}$ be two sequences of complexed valued random variables defined on the same probability space. Assume that $X_n$ converges almost surely to $1$. Then $(X_nY_n)_{n\in\N}$ is tight if and only if $(Y_n)_{n\in\N}$ is tight.  
\end{lem}
\begin{proof}
This is an immediate consequence of the following inequalities:
\begin{align*}
\PP(|X_n Y_n| \ge C)  &\le \PP(|X_n| \ge \sqrt{C}) + \PP (| Y_n| \ge \sqrt{C});
\\
\PP(|Y_n|\ge C)  & \le \PP(|X_nY_n| \ge \sqrt{C}) + \PP(|X_n| \le \frac{1}{\sqrt{C}}). 
\end{align*}
\end{proof}

 For simplifying notation, in what follows, given $n\in\N$ and $X\in\M$, we denote
\begin{align*}
m_n(X): = m_{O(n) \times O(m)} (X).
\end{align*}
\begin{proof}[Proof of Proposition \ref{prop-incl}]
 By Lemma \ref{lem-diag}, it suffices to show that for any $\mu\in \mathscr{L}(\M)$, there exists $s\in\Delta$ such that for any  $\lambda\in\Delta$, we have
 $$\widehat{\mu}(D_\lambda)= \widehat{\mu}_s(D_\lambda).$$
 By definition of $\mathscr{L}(\M)$, there exists $X\in\M$, such that $m_n(X) \Longrightarrow\mu$.  It follows that 
 \begin{align}\label{exist-cv}
 \widehat{\mu}(D_\lambda) = \lim_{n\to\infty} \widehat{m_n(X) }(D_\lambda). 
 \end{align}
 Moreover, the convergence is uniform when $\lambda$ ranges over any compact subsets.   Let $Z^{(n)}$ and $\O$ be two  independent  random matrices sampled uniformly from $O(n)$ and $O(m)$ respectively. We have 
 \begin{align*}
  \widehat{m_n(X) }(D_\lambda) & = \E\Big[\exp\Big(i \tr ( D_\lambda C_m(Z^{(n)} X \O) \Big) \Big]. 
 \end{align*}
 For fixed $n$, by the $O(n)$-invariance of $Z^{(n)}$ and $O(m)$-invariance of $\O$, there exists $s^{(n)} \in \Delta$, such that 
 \begin{align*}
 \E\Big[\exp(i \tr ( D_\lambda C_m(Z^{(n)} X \O) ) \Big] =   \E\Big[\exp\Big(i \tr ( D_\lambda C_m(\sqrt{n}Z^{(n)} \left[\begin{array}{c} D_{s^{(n)}}\\ 0  \end{array}\right] \O) \Big) \Big]. 
 \end{align*}
   {\flushleft \bf Claim:} $\sup_{n\in\N} s_1^{(n)} < \infty$. 

 Assume by contradiction there exists a subsequence $(s_1^{(n_k)})_{k\in\N}$ such that $\lim_{k\to\infty} s_1^{(n_k)}= \infty$.  Using the truncation notation $Z^{(n)}[m]$ introduced in \S \ref{sec-trunc}, we have 
 \begin{align}\label{cv-fourier}
  \widehat{m_n(X) }(D_\lambda)  = \E\Big[\exp(i \tr ( D_\lambda \cdot \sqrt{n}Z^{(n)}[m] \cdot D_{s^{(n)}} \O ) \Big]. 
 \end{align}
  Take now $\lambda = (\lambda_1, 0, \cdots, 0)$. We may assume that the transposition of $Z^{(n)}$ is produced as in \S \ref{sec-haar}, that is, $Z^{(n)}$ is the random matrix obtained by the Gram-Schimidt operation with respect to {\it rows} from a Gaussian random matrix $G_n = [g_{lj}]_{1\le l, j \le n}$.  Then 
 \begin{align*}
  \widehat{m_n(X) }(D_\lambda)  = \E \Big[  i \lambda_1 \sqrt{\frac{n}{\sum_{j=1}^n g_{1j}^2}}  \sum_{j=1}^m g_{1j} s_j^{(n)} \O_{j1}\Big]. 
    \end{align*}
 The uniform convergence \eqref{exist-cv}  on any compact subsets implies that the limit
 \begin{align*}
 \lim_{k\to\infty}\E \Big[  i \lambda_1 \sqrt{\frac{n_k}{\sum_{j=1}^{n_k} g_{1j}^2}}  \sum_{j=1}^m g_{1j} s_j^{(n_k)} \O_{j1}\Big]
 \end{align*}
 exists and the convergence is uniform when $\lambda_1$  ranges over any compact subsets of $[0, \infty)$ and hence by symmetry of the Gaussian distribution, on any compact subset of $\R$. It follows that the following sequence 
 \begin{align*}
 \Big(\sqrt{\frac{n_k}{\sum_{j=1}^{n_k} g_{1j}^2}}  \sum_{j=1}^m g_{1j} s_j^{(n_k)} \O_{j1}\Big)_{k\in\N}
\end{align*}
is tight. By the strong law of large numbers, we have 
$$\sqrt{\frac{n_k}{\sum_{j=1}^{n_k} g_{1j}^2}} \xrightarrow[k\to\infty]{a.s.} 1.$$
Thus we may apply Lemma \ref{lem-2-seq} to conclude that  the following sequence 
 \begin{align*}
 \Big( \sum_{j=1}^m g_{1j} s_j^{(n_k)} \O_{j1}\Big)_{k\in\N}
\end{align*}
is also tight. It follows,  passing to a further subsequence if necessary, that there exists a probability measure $\sigma$ on $\R$ such that 
\begin{align*}
\sum_{j=1}^m g_{1j} s_j^{(n_k)} \O_{j1} \xrightarrow[k\to\infty]{\text{ in distribution}} \sigma. 
\end{align*}
In particular,   for any $\lambda_1 \in \R$,  we have 
\begin{align*}
\widehat{\sigma}(\lambda_1) = \E \Big[i \lambda_1 \sum_{j=1}^m g_{1j} s_j^{(n_k)} \O_{j1} \Big]  = \E\Big[\exp\Big(-\lambda_1\sum_{j=1}^m \frac{(s_j^{(n_k)})^2  (\O_{j1})^2}{2}\Big)\Big]. 
\end{align*}
Hence 
\begin{align*}
0\le & \widehat{\sigma}(\lambda_1)  \le \E\Big[\exp\Big(-\lambda_1 \frac{(s_1^{(n_k)})^2  (\O_{11})^2}{2}\Big)\Big]. 
\end{align*}
Since $\O_{11}\ne 0$ a.s. and by assumption $\lim_{k\to\infty} s_1^{(n_k)} =\infty$, we may apply bounded convergence theorem to conclude that 
\begin{align*}
\widehat{\sigma}(\lambda_1) = 0, \text{ for all $\lambda_1\in\R$. }
\end{align*} 
This contradicts to the fact that $\sigma$ is a probability measure on $\R$. Hence we must have $\sup_{n\in\N} s_1^{(n)} < \infty$. 

Now since $\{s \in \Delta| s_1 \le \sup_{n\in\N} s_1^{(n)} \}$ is compact, we may assume that there exists a subsequence $(s^{(n_k)})_{k\in\N}$ converges to a point $s^{(\infty)}\in\Delta$. Taking Proposition \ref{prop-trun} into account, the equalities \eqref{exist-cv} and \eqref{cv-fourier} now imply 
\begin{align*}
\widehat{\mu}(D_\lambda) &= \lim_{k\to\infty}  \E\Big[\exp(i \tr ( D_\lambda \cdot \sqrt{n_k}Z^{(n_k)}[m] \cdot D_{s^{(n_k)}} \O ) \Big] 
\\
& =  \E\Big[\exp(i \tr ( D_\lambda \cdot G_m \cdot D_{s^{(\infty)}} \O ) \Big] 
\\
&=  \E\Big[\exp(i \tr ( D_\lambda \cdot C_m(G D_{s^{(\infty)}} \O) ) \Big]. 
\end{align*} 
By definition of the probability measure $\mu_{s^{(\infty)}}$, we get 
\begin{align*}
\widehat{\mu}(D_\lambda) = \widehat{\mu}_{s^{(\infty)}}(D_\lambda), \text{for all $\lambda\in\Delta$. }
\end{align*}
Hence the proof of Proposition \ref{prop-incl} is completed.
\end{proof}

%%%%%%%%%
%%%%%%%%%
%%%%%%%%%
%%%%%%%%%
%%%%%%%%%
%%%%%%%%%
%%%%%%%%%
%%%%%%%%%
%%%%%%%%%
%%%%%%%%%
%%%%%%%%%
%%%%%%%%%
%%%%%%%%%
%%%%%%%%%

%%%%%%%%%
%%%%%%%%%
%%%%%%%%%
%%%%%%%%%
%%%%%%%%%
%%%%%%%%%
%%%%%%%%%
%%%%%%%%%
%%%%%%%%%
%%%%%%%%%
%%%%%%%%%
%%%%%%%%%
%%%%%%%%%
%%%%%%%%%
%%%%%%%%%

%
%
%\bibliography{mybib}
%\bibliographystyle{plain}
%

\def\cprime{$'$} \def\cydot{\leavevmode\raise.4ex\hbox{.}} \def\cprime{$'$}

\end{document}